\documentclass{amsart}
\usepackage{amssymb,enumerate,amscd,amsthm,amsmath,graphics}

 \newtheorem{theorem}{Theorem}[section]
 \newtheorem{corollary}[theorem]{Corollary}
 \newtheorem{lemma}[theorem]{Lemma}
 \newtheorem{proposition}[theorem]{Proposition}
 \theoremstyle{definition}
 
 \theoremstyle{remark}

 \numberwithin{equation}{section}
 
\begin{document}

\title[Commuting powers and exterior degree of finite groups]
 {Commuting powers and exterior degree of finite groups}

\author[P. Niroomand]{Peyman Niroomand}
\address{School of Mathematics and Computer Science\\
Damghan University of Basic Sciences\\
Damghan, Iran}
\email{p$\_$niroomand@yahoo.com}

\author[R. Rezaei]{Rashid Rezaei}
\address{Department of Mathematics, Faculty of
Mathematical Sciences\\
Malayer University\\
Post Box: 657719--95863, Malayer, Iran}
\email{ras$\_$rezaei@yahoo.com}

\author[F.G. Russo]{Francesco G. Russo}
\address{Department of Mathematics\\
University of Palermo, via Archirafi 34, 90123, Palermo, Italy
} \email{francescog.russo@yahoo.com}





\begin{abstract}
In [P. Niroomand, R. Rezaei, On the exterior degree of finite groups, Comm. Algebra 39 (2011), 335--343] it is introduced a group invariant, related to the number of elements  $x$ and $y$ of a finite group $G$, such that $x \wedge y= 1_{_{G \wedge G}}$ in the exterior square $G \wedge G$ of $G$. This number gives restrictions on the Schur multiplier of $G$ and, consequently, large classes of groups can be described. In the present paper  we generalize  the previous investigations on the topic, focusing on the number of elements of the form $h^m \wedge k$  of $H \wedge K$ such that $h^m \wedge k= 1_{_{H \wedge K}}$, where $m\ge1$ and $H$ and $K$ are arbitrary subgroups of $G$.
\end{abstract}

\maketitle







\section{Non--abelian tensor product, homological algebra and commutativity degree}

All the groups, which are considered in the paper, are supposed to be finite. Some technical notions of homological algebra should be recalled from \cite{e1,e2,e3} in order to formulate our topic of investigation in an appropriate way.

For any group $G$ we can construct functorially a \textit{classifying space} $B(G)$ with the following properties.
\begin{itemize}
\item[1)] The topological space $B(G)$ is a connected CW-complex.
\item[2)] The fundamental group $\pi_1(B(G))$ of $B(G)$ is isomorphic to $G$.
\item[3)] The higher homotopy groups $\pi_n(B(G))$ are trivial for $n\ge 2$.
\end{itemize}
The singular homology groups of any space $X$, with coefficients in the abelian
group $\mathbb{Z}$, will be denoted by $H_n(X)$. Since the homology groups $H_n(B(G))$ depend
only on the group $G$, we can write $H_n (G) = H_n (B(G))$, for all $n\ge 0$.
For each normal subgroup $H$ in $G$ we functorially construct a space $B(G,H)$
as follows. The natural homomorphism $G\rightarrow G/H$ induces a map $f : B(G) \rightarrow
B(G/H)$. Let $M(f)$ denote the mapping cylinder of this map. Note that $B(G)$
is a subspace of $M(f)$, and that $M(f)$ is homotopy equivalent to $B(G/H)$. We
take $B(G,H)$ to be mapping cone of the cofibration $B(G)\rightarrow M(f)$.
The cofibration sequence $B(G)\rightarrow M(f) \rightarrow B(G,H)$ yields a natural long
exact homology (Mayer--Vietoris) sequence
$\ldots \rightarrow H_{n+1}(G/H) \rightarrow H_{n+1}(B(G,H)) \rightarrow H_n(G) \rightarrow H_n(G/H) \rightarrow \ldots$
for $n\ge 0$. It can be shown that
$H_1 (B(G,H)) = 0$ and $H_2(B(G,H))\simeq H/[H,G]$. The classifying space $B(F)$ of a free group $F$ is one-dimensional, and so $H_n (F) = 0$
for $n\ge 2$ and it is easy to check that 
$H_1 (G)\simeq G/[G,G]=G^{ab}$ and $H_2(G)\simeq \ker  \psi \simeq M(G)$,
where $G = F/R$ is a presentation of $G$ from $F$ and $R$ and $\psi: R/[R,F] \rightarrow F/[F,F]$ is a natural homomorphism and $M(G)$ is the \textit{Schur multiplier} of $G$. Now it is meaningful to define the \textit{Schur multiplier of the pair of groups} $(G,H)$ as the set
$M(G,H) = H_3 (B(G,H))$. We can generalize more. By a \textit{triple} we mean a group $G$ with two normal subgroups $H$ and $K$. A
\textit{homomorphism of triples} $(G,H,K) \rightarrow (G',H',K')$ is a group homomorphism
$G \rightarrow G'$ that sends $H$ into $H'$ and $K$ into $K'$. The \textit{Schur multiplier of the triple}
$(G,H,K)$ is a functorial abelian group $M(G,H,K)$ whose principal feature is
a natural exact sequence
$H_3(G,H) \rightarrow H_3(G/H, HK/K ) \rightarrow M(G,H,K) \rightarrow M(G,K)
\rightarrow M(G/H, HK/H ) \rightarrow H \cap K/[H\cap K,G][H,K]
\rightarrow K/[K,G]  \rightarrow K H/H[K,G] \rightarrow 0$
in which, by definition, $H_3(G,H)=H_4(B(G,H))$. The definition of $M(G,H,K)$ is in terms of the mapping cone $B(G,H,K)$
of the canonical cofibration $B(G,K) \rightarrow B(G/K,HK/H)$. An analogy with the case of pairs allows us to define
$M(G,H,K) = H_4(B(G,H,K))$.

The Schur multiplier of a triple is related to an important construction, which we recall as in \cite[Section 3]{e3} and \cite{bjr,bl}. A group  $G$ acts by conjugation on its normal subgroups $H$ and $K$ via the rule $^gx=gxg^{-1}$, for $g$ in $G$ and $x$ in $H$ or $K$, and the \textit{exterior product} $H \wedge K$ is defined as the group generated by the symbols $h \otimes k$, subject to the relations:
\begin{equation}hh'\otimes k=(~^hh'\otimes ~^hk) \ (h\otimes k), \ \ \ \ 
kk'\otimes h=(k\otimes h) \ (~^kh \otimes ~^kk'), \ \ \ \
y \otimes y=1,\end{equation}
where $h, h' \in H$, $k, k' \in K$ and $y \in H \cap K$. Briefly, $h \wedge k$ denotes $h \otimes k$ satisfying all the above relations. The map $\kappa' : h\wedge k \in H \wedge K \mapsto [h,k]=hkh^{-1}k^{-1} \in [H,K]$ turns out to be a group epimorphism, whose kernel $\ker \kappa' $ is abelian. Furthermore, $\ker \kappa' \simeq M(G,H,K)$ whenever $G=HK$ (see \cite[Theorem 6.1]{e3}). Omitting the relation $y \otimes y =1$, it is similarly defined the \textit{non--abelian tensor product} $H \otimes K$ of $H$ and $K$. By analogy, the map $\kappa : h\otimes k \in H \otimes K \mapsto [h,k]=hkh^{-1}k^{-1} \in [H,K]$ turns out to be a group epimorphism, whose kernel $\ker \kappa=J(G,H,K)$ is again abelian. We note that $J(G,H,K)$ is related to the fundamental group of a covering space and has significant interest in algebraic topology (see \cite{bjr,bl,e1,e2,e3}). 

The above information are summarized below, where $G=HK$ (with $H$ and $K$ normal in $G$).
\begin{equation}\label{diag}\begin{CD}
1 @>>>J(G,H,K) @>>>H \otimes K @>\kappa>>[H,K]@>>>1\\
@. @VVV @VVV @|\\
1 @>>> M(G,H,K) @>>>H \wedge K @>\kappa'>> [H,K] @>>>1.\\
\end{CD}
\end{equation}
From the results in \cite{bjr,bl,e1,e2,e3}, \eqref{diag} is commutative with central extensions as rows and natural epimorphisms $\pi: h \otimes k \in J(G,H,K) \mapsto h \wedge k \in M(G,H,K)$, $\epsilon: h\otimes k \in  H \otimes K \mapsto h \wedge k \in H \wedge K$ as columns. Of course, if $G=H=K$, then $M(G)$ is the \textit{Schur multiplier} of $G$, $H \otimes K = G \otimes G$ is the \textit{non--abelian tensor square} of $G$ and,  in particular, $G^{ab} \otimes_\mathbb{Z} G^{ab}$ is the usual tensor square of an abelian group.

It may be helpful to recall that the actions of $H$ on $K$ induce an action $a \in H*K \ \longmapsto \ ^a(h\otimes k) = \ ^ah \ \otimes \ ^ak \in H \otimes K$, which allows us to see $H \otimes K$  as a suitable homomorphic image of the central product $H * K$.   In this context, if $x \in G$, the \textit{exterior centralizer} of $x$ in $G$ is the set $C_G^\wedge(x)=\{a\in G  \ | \ a \wedge x=1_{_{G \wedge G}}\}$, which turns out to be a subgroup of $G$ and the \textit{exterior center} of $G$ is the set $Z^\wedge(G)=\{g \in G \ | \ 1_{_{G \wedge G}}=g \wedge y \in G \wedge G, \forall y \in G\}={\underset{x \in G}\bigcap} C_G^\wedge(x)$ which  is a subgroup of the  center $Z(G)$ of $G$. Further details can be found in \cite{e1,e2,pf,nr}. Very briefly, we mention that the interest in studying $C_G^\wedge(x)$ and $Z^\wedge(G)$ is due to the fact that they allow us to decide whether $G$ is a \textit{capable group} or not, that is, whether $G$ is isomorphic to $E/Z(E)$ for some group $E$ or not. \cite{bfs} and \cite[Chapter 21]{b} illustrate that capable groups are well--known and classified.

Now we recall from \cite{dn1,dn2,elr,err,gr,l1,l2,r} that the \textit{commutativity degree} of $G$ is the ratio
\begin{equation}d(G) = \frac{|\{(x, y) \in G \times G  \ | \ [x, y] = 1\}|}{|G|^2} = \frac{1}{|G|^2} \underset{x \in G} {\sum} |C_G(x)|= \frac{k(G)}{|G|},\end{equation}
where $k(G)$ is the \textit{number of the $G$--conjugacy classes} $[x]_G=\{x^g \ | \ g \in G\}$ that constitute $G$. There is a  wide production on $d(G)$ and its generalizations in the last decades. For instance, given an arbitrary subgroup $H$ of $G$, it was introduced in \cite{elr} the \textit{$n$-th relative  nilpotency degree} of $G$ 
\begin{equation}d^{(n)}(H, G) = \frac{|\{(h_1,\ldots, h_n, g) \in H^n \times G  \ | \  [h_1,\ldots,h_n,g] = 1\}|}{|H|^n \ |G|}	= \frac{1}{|H|^n \ |G|} \underset{h_1,\ldots, h_n \in H} {\sum} |C_G([h_1,\ldots,h_n])|.\end{equation} It is clearly a generalization of $d(G)$, and, in case $n=1$, it was proposed the further generalization \begin{equation}d(H,K) = \frac{|\{(h,k) \in H \times K  \ | \ [h,k] = 1\}|}{|H| \ |K|}= \frac{1}{|H| \ |K|} \sum_{h \in H} |C_K(h)|=\frac{k_K(H)}{|H|}\end{equation}  in \cite{dn1}, where $H$ is a normal subgroup of $G$, $K$ is an arbitrary subgroup of $G$ and $k_K(H)$ is the \textit{number of the $K$--conjugacy classes} $[h]_K=\{h^k \ | \ k \in K\}$ that constitute $H$.

We will focuse on a recent contribution in \cite{pr}, where it is introduced the
\textit{exterior degree} of $G$  \begin{equation}d^\wedge(G)=\frac{|\{(x,y)\in G \times G \ | \  x \wedge y =1_{_{G \wedge G}}\}|}{|G|^2},\end{equation} which can be written by \cite[Lemma 2.2]{pr} as \begin{equation}d^\wedge(G)=\frac{1}{|G|} \sum^{k(G)}_{i=1} \frac{|C^\wedge_G(x_i)|}{|C_G(x_i)|}.\end{equation} In analogy, given two arbitrary subgroups $H$ and $K$ of $G$, we  define for  $m\ge1$ the \textit{$m$-th relative exterior degree} of $H$ and $K$ in $G$ \begin{equation}d^\wedge_m (H,K)=\frac{|\{(h,k)\in H \times K \ | \  h^m \wedge k =1_{_{H \wedge K}}\}|}{|H| \  |K|}.\end{equation} In particular, $d^\wedge_m(G)=d^\wedge_m(G,G)$ is the  \textit{$m$-th exterior degree} of $G$ and, of course, $d^\wedge_1(G,G)=d^\wedge(G)$ so that it is meaningful to generalize the bounds in \cite{pr}. We also note that for $H=G$ and $m=1$ there are results on $d^\wedge(G,K)$ in  \cite{nr}.  While the commutativity degree represents the probability that two randomly picked elements of  $G$ are commuting, the $n$-th relative nilpotency degree is a variation on this theme. By analogy with the operator $\wedge$, the $m$-th relative exterior  degree is a variation on the theme of the exterior degree, involving the powers of  $x$ and the single element $y$. We will study the effects of $d^\wedge_m(H,K)$ on the structure of $G$ in the successive sections.

\section{Basic properties}

An immediate observation is that we may rewrite $d^\wedge_m (H,K)$ as:
\begin{equation}d^\wedge_m (H,K)= \frac{1}{|H| \ |K| }  \sum_{h\in H} |C^\wedge_K(h^m)|.
\end{equation} Assume that $H$ is normal in $G$ and  $C_1 \ldots, C_{k_K(H)}$ are the $K$--conjugacy classes that constitute $H$. It follows that 
\begin{equation}\label{formula}|H| \ |K| \ d^\wedge_m (H,K)=  \sum_{h\in H} |C^\wedge_K(h^m)|=\sum^{k_K(H)}_{i=1} \sum_{h \in C_i} |C^\wedge_K(h^m)|=\sum^{k_K(H)}_{i=1} |K:C_K(h_i)| \ |C^\wedge_K(h^m_i)|\end{equation}
\[= \sum^{k_K(H)}_{i=1}  \frac{|K|}{|C_K(h^m_i)|} \ \frac{|C_K(h^m_i)|}{|C_K(h_i)|} \ |C^\wedge_K(h^m_i)|
=|K| \ \sum^{k_K(H)}_{i=1} \left( \frac{|C_K(h^m_i)|}{|C_K(h_i)|} \right) \ \frac{|C^\wedge_K(h^m_i)|}{|C_K(h^m_i)|}=|K| \ \sum^{k_K(H)}_{i=1}  \alpha(m,i) \frac{|C^\wedge_K(h^m_i)|}{|C_K(h^m_i)|},\]
where $\alpha(m,i)$ is the index of $|C_K(h_i)|$ in  $|C_K(h^m_i)|$ and then a natural number. The assumption that $H$ has to be normal in $G$ is done in order to have an entire conjugacy class which is fixed under the action of $K$ on $H$. It may be helpful for the rest of the paper to define the group
\begin{equation}L(m,i;h,K) =  \frac{C_K(h^m_i)}{C^\wedge_K(h^m_i)}.\end{equation}

\begin{lemma} \label{l:1}
Let $H$ be a normal subgroup of a group $G$ and $K$ be a subgroup of $G$. Then 
\begin{equation}d^\wedge_m (H,K)=  \frac{1}{|H|} \sum^{k_K(H)}_{i=1}  \alpha(m,i) \frac{|C^\wedge_K(h^m_i)|}{|C_K(h^m_i)|}=\frac{1}{|H|} \ \sum^{k_K(H)}_{i=1} \frac{\alpha(m,i)}{|L(m,i;h,K)|}.\end{equation}
In particular, if $G=HK$ and $K$ is normal in $G$, then $L(m,i;h,K)$ is isomorphic to a subgroup of $M(G,H,K)$.
\end{lemma}

\begin{proof} The first part follows from \eqref{formula}. Now assume that $G=HK$ for $H$ and $K$  normal in $G$. The exact sequence \eqref{diag} implies that for all $i=1,\ldots, k_K(H)$ 
the map $x\in C_K(h_i^m) \mapsto h_i^m\wedge x \in M(G,H,K)$ 
 is a homomorphism of groups. On another hand,  its kernel is $C^\wedge_K(h_i^m)$, and, consequently, $L(m,i;h,K) $ is isomorphic to a subgroup of $ M(G,H,K)$.
\end{proof}

The sequence $d^\wedge_m(H,K)$ is monotone in the sense of the next result. We should do an assumption on $m$ of being of prime power order. This will be necessary (but not sufficient) to have the subgroup lattice of a cyclic group which is a chain.

\begin{proposition}\label{p:1} Let $H$ and $K$ be subgroups of $G$ and $p$ be a  prime divisor of $|H|$. Then there exists an integer $r\ge 1$ such that \begin{equation} d^\wedge_{p^{r-1}}(H,G)  \ge  d^\wedge_{p^{r-1}}(H,K) \ge d^\wedge_{p^{r-2}}(H,K) \ge \ldots   \ge d^\wedge_p(H,K)\ge d^\wedge(H,K).\end{equation} 
\end{proposition}

\begin{proof} Let $h\in H$ be of order $p^r$ for some integer $r\ge 1$. Then 
 $\{1\}=\langle h^{p^r} \rangle \le \langle h^{p^{r-1}} \rangle \le \ldots \le \langle h \rangle$ implies $C^\wedge_K(\{1\})= K  \ge C^\wedge_K(h^{p^{r-1}})=C^\wedge_K(\langle h^{p^{r-1}}\rangle) \ge \ldots \ge C^\wedge_K(h^p)=C^\wedge_K(\langle h^p \rangle) \ge C^\wedge_K(h)=C^\wedge_K(\langle h \rangle)$. Therefore
\begin{equation}  \sum_{h \in H} |C^\wedge_K(h)|  \le \sum_{h \in H} |C^\wedge_K(h^p)|  \le  \ldots \le \sum_{h \in H} |C^\wedge_K(h^{p^{r-1}})|,\end{equation}
from which we deduce
\begin{equation}d^\wedge(H,K) \le d^\wedge_p(H,K) \le   \ldots  \le d^\wedge_{p^{r-1}}(H,K).  \end{equation}
On another hand, \begin{equation}|H| \ |G| \ d^\wedge_{p^{r-1}}(H,G)=  \sum_{h \in H} |C^\wedge_G(h^{p^{r-1}})| \ge \sum_{h\in H} |C^\wedge_K(h^{p^{r-1}})|=|H| \ |K| \ d^\wedge_{p^{r-1}}(H,K).\end{equation}
\end{proof}

Among groups with trivial Schur multiplier there are important classes of groups. For instance, a cyclic group $C=\langle c \rangle$ has $|M(C)|=1$ by    \cite[Lemma 21.1]{b}; a metacyclic group of the form  $D=\langle a,b \ | \ a^{p^n}=b^p=1, b^{-1}ab=a^{1+p^{n-1}}\rangle$ (where $n\ge 3$ if $p=2$) has also $|M(D)|=1 $ by  \cite[Theorem 1.2 and Lemma 21.2]{b};  finally, looking at \cite{aavv}, several sporadic simple groups have trivial Schur multiplier. In our context, we are interested to see what happens to $d^\wedge_m(H,K)$ when $M(G,H,K)$ is trivial. Immediately, we find the next consequence of  Lemma \ref{l:1}.

\begin{corollary}\label{c:1extra} Let $G=HK$ for two normal subgroups $H$ and $K$ of $G$ with $H$ of exponent $p^r-1$ for some $r\ge1$ and some prime $p$. If $M(G,H,K)$ is trivial, then $\alpha(p^r,i)=|L(p^r,i;h,K)|=1$.
 \end{corollary}
\begin{proof}
By Lemma \ref{l:1}, $|L(p^r,i;h,K)|=1$. The fact that $H$ has exponent $p^r-1$ implies  $h^{p^r-1}_i=1$, that is, $h^{p^r}_i=h_i$ for all $i=1,\ldots, k_K(H)$, and then $\alpha(p^r,i)=1$. 
\end{proof}

We can refine the condition at infinity of $r$, by looking at Proposition \ref{p:1}, and we have the following result.

\begin{corollary}\label{c:1} Let $H$  be a normal  subgroup of a group $G$, $K$ a subgroup of $G$ and $p$ a  prime divisor of $|H|$. Then ${\underset{r \rightarrow 0} \lim} \ d^\wedge_{p^r}(H,K) = d^\wedge(H,K)$. 
Furthermore, if $ {\underset{r \rightarrow \infty} \lim} \frac{\alpha(p^r,i)}{|L(p^r,i;h,K)|}=1$ and the action of $K$ on $H$ induces just one orbit, then ${\underset{r \rightarrow \infty} \lim} \ d^\wedge_{p^r}(H,K) \le \frac{1}{p}$. In particular, $d(H,K)={\underset{r \rightarrow \infty} \lim} \ d^\wedge_{p^r}(H,K) = \frac{1}{p}$, provided that $|H|=p$.
\end{corollary}

\begin{proof} The first part of the result follows from Proposition \ref{p:1}. 

Lemma \ref{l:1} and the assumptions imply 
\begin{equation}\label{c}{\underset{r \rightarrow \infty} \lim}  d^\wedge_{p^r}(H,K) = {\underset{r \rightarrow \infty} \lim}  \frac{1}{|H|}  \sum^{k_K(H)}_{i=1} \frac{\alpha(p^r,i)}{|L(p^r,i;h,K)|}= \frac{1}{|H|}  {\underset{r \rightarrow \infty} \lim} \sum^{k_K(H)}_{i=1} \frac{\alpha(p^r,i)}{|L(p^r,i;h,K)|}\end{equation} 
\[= \frac{1}{|H|} \ \sum^{k_K(H)}_{i=1} {\underset{r \rightarrow \infty} \lim} \ \frac{\alpha(p^r,i)}{|L(p^r,i;h,K)|}= \frac{ k_K(H)}{|H|}=d(H,K).\]
The choice of $p$ implies $\frac{1}{|H|}\le \frac{1}{p}$ and therefore ${\underset{r \rightarrow \infty} \lim} \ d^\wedge_{p^r}(H,K) \le \frac{k_K(H)}{p}$. In particular, if the action of $K$ on $H$ induces just one orbit, then $k_K(H)$ is just one and so ${\underset{r \rightarrow \infty} \lim} \ d^\wedge_{p^r}(H,K) \le \frac{1}{p}$. The rest  follows  clearly from \eqref{c}.
\end{proof}

With respect to direct products there is a sort of natural splitting for $d^\wedge_m(H,K)$ and this is shown below.

\begin{proposition}\label{p:2}
If $A,B,C,D$ are subgroups of a group $G$ such that $(|A|,|B|)=(|C|,|D|)=1$, then \begin{equation}d^\wedge_m(A\times B, C \times D)= d^\wedge_m(A,C) \cdot  d^\wedge_m(B,D).\end{equation}
\end{proposition}

\begin{proof} \begin{equation}|A\times B| \ |C \times D| \ d^\wedge_m(A\times B, C \times D) 
= |A| \ |B| \ |C| \ |D| \ d^\wedge_m(A\times B, C \times D)\end{equation}\begin{equation}=  \sum_{(a,b) \in A \times B} |C^\wedge_{C \times D}((a^m,b^m))|= \left(\sum_{a \in A} |C^\wedge_C(a^m)|\right) \ \left(\sum_{b \in B} |C^\wedge_D(b^m)|\right)
= |A| \ |C| \ d^\wedge_m(A,C) \ |B| \ |D| \ d^\wedge_m(B,D).\end{equation} 
\end{proof}

In particular, \cite[Lemma 2.10]{pr} can be found as a special case of the previous result. Another general property is encountered when we go to form quotients and for $m=1$  it can be found in \cite[Proposition 2.6]{pr}. Before to describe it, we introduce the set $Z^\wedge(H,K)=\{h \in H \ | \ h \wedge k=1{_{H \wedge K}}  \  \forall k \in K\}$, where $H$ and $K$ are normal subgroups of $G$, acting upon each other by conjugation.   $Z^\wedge(H,K)$  is largely described in \cite{nr} when $G=H$ and it is easy to check  that $Z^\wedge(H,K)$ is a subgroup of $H$, and, in particular, $Z^\wedge(G,G)=Z^\wedge(G)$.

\begin{proposition}\label{p:3}If $H$ and $K$ are two subgroups of $G$ containing a normal subgroup $N$ of $G$, then
$d^\wedge_m(H,K) \le d^\wedge_m(H/N, K/N).$ The equality holds, if  $N\subseteq Z^\wedge(H,K)$.
\end{proposition}

\begin{proof}
\begin{equation}|H| \ |K| \ d^\wedge_m (H,K)=  \sum_{h \in H} |C^\wedge_K(h^m)| = \sum_{hN\in H/N} \sum_{n \in N}|C^\wedge_K(h^mn)|=\sum_{hN\in H/N} \sum_{n \in N} \frac{|C^\wedge_K(h^mn)N|}{|N|} \ |C^\wedge_K(h^mn)\cap N|\end{equation}
\begin{equation}\le \sum_{hN\in H/N} \sum_{n \in N}|C^\wedge_{K/N}(h^mN)| \ |C^\wedge_K(h^mn)\cap N|=\sum_{hN\in H/N} |C^\wedge_{K/N}(h^mN)| \ \sum_{n\in N}|C^\wedge_K(h^mn)\cap N|\end{equation}\begin{equation}\le |N|^2 \sum_{hN \in H/N}|C^\wedge_{K/N}(h^mN)|  = |H| \ |K| \ d^\wedge(H/N,K/N).\end{equation}
We find always an exact sequence 
\begin{equation}\label{natural}\begin{CD}
1 @>>>N\wedge K  @>\varphi>>H \wedge K @>\epsilon>>(H/N)\wedge (K/N)@>>>1
\end{CD}
\end{equation}
where $\iota : n \in N \mapsto \iota(n) \in H$ is the natural embedding of $N$ into $H$, $\varphi: n\wedge k \in N \wedge K \mapsto \iota(n) \wedge h \in H \wedge K$ and $\epsilon: h \wedge k \in H \wedge K \mapsto hN \wedge kN \in (H/N)\wedge (K/N)$ is induced by the natural epimorphisms of $H$ onto $H/N$ and of $K$ onto $K/N$.
If $N\subseteq Z^\wedge(H,K)$, then $\mathrm{Im} \ \varphi=1_{_{N \wedge K}}$ and \eqref{natural} implies $H/N \wedge K/N\simeq H \wedge K$ so that $|N|^2 \ |\{(hN,kN)\in H/N \times K/N \ | \ h^mN\wedge kN=1_{_{(H/N) \wedge (K/N)}}\}|=|\{(h,k)\in H \times K \ | \ h^m\wedge k=1_{_{H \wedge K}}\}|$, hence $d^\wedge_m(H,K) = d^\wedge_m(H/N, K/N)$.
\end{proof}

A general restriction is the following.

\begin{theorem}\label{t:1}Let $G=HK$ for two normal subgroups $H$ and $K$. Then for all $m\ge 1$
\begin{equation}\label{et1}\beta(m) \ \frac{d(H,K)}{|M(G,H,K)|}\le  d^\wedge_m(H,K) \le  \gamma(m) \ d(H,K),\end{equation}
where $\beta(m)=\mathrm{min}\{\alpha(m,i) \ | \ i=1,\ldots, k_K(H) \}$ and
$\gamma(m)=\mathrm{max}\{\alpha(m,i) \ | \ i=1,\ldots, k_K(H) \}$.
\end{theorem}

\begin{proof}
Keeping in mind Lemma \ref{l:1} and  noting that $C_K(h^m_i)/C^\wedge_K(h^m_i)$ is isomorphic to a subgroup of $M(G,H,K)$, we have $|C^\wedge_K(h^m_i)|/|C_K(h^m_i)|\ge 1/|M(G,H,K)|$. Therefore
\begin{equation}d^\wedge_m (H,K)=  \frac{1}{|H|} \sum^{k_K(H)}_{i=1}  \alpha(m,i) \ \frac{|C^\wedge_K(h^m_i)|}{|C_K(h^m_i)|} \ge
\frac{\beta(m) \ k_K(H)}{|H| \ |M(G,H,K)|}
=\beta(m) \ \frac{\ d(H,K)}{|M(G,H,K)|}. \end{equation} 
On another hand,  again Lemma \ref{l:1} implies
\begin{equation}
d^\wedge_m (H,K)=  \frac{1}{|H|} \sum^{k_K(H)}_{i=1}  \alpha(m,i) \ \frac{|C^\wedge_K(h^m_i)|}{|C_K(h^m_i)|}
\le \gamma(m)\ \frac{k_K(H)}{|H|}= \gamma(m) \ d(H,K).
\end{equation}
\end{proof}

In \cite{e1, pf, pr} it was noted that a group $G$ such that $Z^\wedge(G)=Z(G)$ has strong structural restrictions; among these it was noted in \cite{pr} that $d^\wedge_1(G)=d^\wedge(G)=d(G)$. We find something of similar in the next result.

\begin{corollary}\label{c:2} Let $G=HK$ for two normal subgroups $H$ and $K$. If $M(G,H,K)$ is trivial and $H$ has exponent $m-1$, then  $d^\wedge_m(H,K)=d(H,K)$ for all $m\ge 1$.
\end{corollary}

\begin{proof} Since $M(G,H,K)=1$ is trivial, the lower bound in \eqref{et1} is reduced to $\beta(m) \ d(H,K)$.  $H$ has exponent $m-1$ and then, using the notations of Lemma \ref{l:1}, $h^{m-1}_i=1$, that is, $h^m_i=h_i$, for all $i=1,\ldots, k_K(H)$. Consequently, $\alpha(m,i)=\alpha(m)=\beta(m)=\gamma(m)=1$.  Then \eqref{et1} becomes $d(H,K)\le d^\wedge_m(H,K)\le d(H,K)$ and the result follows. 
\end{proof}

Another consequence of Theorem \ref{t:1} is related to Proposition \ref{p:1}.

\begin{corollary}\label{c:3} Let $G=HK$ for two normal subgroups $H$ and $K$ and $p$ be a prime divisor of $|H|$. Then there exists an integer $r\ge 1$ such that
\begin{equation} \frac{\gamma(p^{r-1})}{p} \ k_K(H)  \ge   d^\wedge_{p^{r-1}}(H,K) \ge  d^\wedge_{p^{r-2}}(H,K)\ge \ldots 
 \ge \beta(p^{r-1}) \ \frac{\ d(H,K)}{|M(G,H,K)|}.\end{equation} 
\end{corollary}

\begin{proof} It is enough to apply Theorem \ref{t:1} and Proposition \ref{p:1}.
\end{proof}

In a certain sense, Corollary \ref{c:1}  continues to be true without restrictions on $m$. This is illustrated below.

\begin{proposition}\label{p:3} Let $G=HK$ for two normal subgroups $H$ and $K$ such that $[H,K]\not=1$. Then $d^\wedge_m(H,K)\le \gamma(m) \ \frac{2p-1}{p^2}$, where  $p$ is the smallest  prime dividing $|G|$ and $|K|$. In particular, if $H$ has exponent $m-1$, then $d^\wedge_m(H,K)\le \frac{2}{p}$.
\end{proposition}

\begin{proof} From the choice of $p$ we deduce $|C^\wedge_K(H)| \le |C_K(H)| \le \frac{|K|}{p}$. Now $|C_K(H)|-|C^\wedge_K(H)|\le \frac{|K|}{p}$. The upper bound in Theorem \ref{t:1} allows us to continue as in \cite[Corollary 3.9]{dn1} and so
\begin{equation} d^\wedge_m (H,K)\le \gamma(m) \ d(H,K) \le \gamma(m) \frac{2p-1}{p^2}.\end{equation}
In particular, if $H$ has exponent $m-1$, then the argument in Corollary \ref{c:2} implies $\gamma(m)=1$, hence \begin{equation}\frac{2p-1}{p^2}=\frac{2}{p}-\frac{1}{p^2} \le \frac{2}{p}.\end{equation}
\end{proof}

The following result justifies the interest  for the numerical restrictions on $d^\wedge_m(H,K)$, which have been the subject of most of the previous bounds. These allow us to describe the position of some subgroups in the whole group, when we consider some special values of the $m$-th relative exterior degree. 

\begin{corollary}\label{c:4}  Let $H$ be a normal subgroup of $G$ of exponent $m-1$ and $K$ be a normal subgroup of $G$ such that $G=HK$ and  $M(G,H,K)$ is trivial.  If $d^\wedge_m(H,K) = \frac{2p-1}{p^2}$ for some prime $p$, then $p$ divides $|G|$. If $p$ is the smallest prime divisor of $|G|$, then $|H:C_H(K)|=|K:C_K(H)|=p$ and, hence, $H \not= K$. In particular, if $d^\wedge_m(H,K)=\frac{3}{4}$, then $|H:C_H(K)|=|K:C_K(H)|=2$.
\end{corollary}

\begin{proof} Corollary \ref{c:2} implies  $d^\wedge_m(H,K)=d(H,K)$ for all $m\ge 1$. The rest follows from \cite[Proposition 3.1]{dn1}.
\end{proof}

\section{Dihedral groups and generalized quaternion groups}

Thanks to the results in Section 2 and to those in \cite{elr, err, l1, l2, pr, nr}, we want to have a closer look at the class of dihedral groups and at that of generalized quaternion groups. Their structure  is described in \cite[Theorem 1.2]{b}: these groups possess a cyclic group of index 2 and are metacyclic. We will be quite general and recall that \begin{equation}D_{2n}=C_n \rtimes C_2=\langle a, b \ | \ a^n=b^2=1, b^{-1}ab=a^{-1}\rangle \end{equation}
is the \textit{dihedral group of order} $2n$, where $n\ge1$.  Assume $D_{2n}=\{1,a,a^2,\ldots,a^{n-1},b,ab,a^2b,\ldots,a^{n-1}b\}$ and  $t=\gcd(m,n)$. Since $Z^{\wedge}(G)=1$ and $(a^{\frac{in}{t}})^{m}=1$ for $0\leq i\leq t-1$,  for $t$
elements of $D_{2n}$ we have $|C^\wedge_{D_{2n}}(x^m)|=2n$ and for $n-t$ elements $|C^\wedge_{D_{2n}}(x^m)|=n$. Now, if $m$ is odd,
then $|C^\wedge_{D_{2n}}((a^jb)^m)|=2$ for $0\leq j\leq n-1$ and so $d^\wedge_m(D_{2n})=\frac{n^2+nt+2n}{4n^2}.$
If $m$ is even, then $(a^jb)^m=1$ and so $|C^\wedge_{D_{2n}}((a^jb)^m)|=2n$, therefore  $d^\wedge_m(D_{2n})=\frac{3n^2+nt}{4n^2}.$
Summarizing,
\begin{equation}\label{d}d^\wedge_m(D_{2n})= \left\{ \begin{array}{lcl}\frac{3n+\gcd(m,n)}{4n},&\,\,& \mathrm{if} \ m \ \mathrm{ is \ even},\\
\frac{n + \gcd(m,n) + 2}{4n},&\,\,& \mathrm{if} \ m \ \mathrm{is \ odd}. \end{array} \right.
\end{equation}
A similar computation can be made for
\begin{equation}
Q_n= \langle a,b \ | \ a^n=b^2=(ab)^2 \rangle,
\end{equation}
which is called \textit{generalized quaternion group of order} $4n$. Here, as done for $D_{2n}$, we find that
\begin{equation}\label{q}d^\wedge_m(Q_n)= \left\{ \begin{array}{lcl}\frac{3n +  \gcd(m,n)}{4n},&\,\,& \mathrm{if} \ m \ \mathrm{ is \ even},\\
\frac{n + \gcd(m,n) + 2}{4n},&\,\,& \mathrm{if} \ m \ \mathrm{is \ odd}. \end{array} \right.
\end{equation}
From \cite[Examples 3.1 and 3.2]{pr},    $d^\wedge(D_{2n})=d(D_{2n}) =d^\wedge(Q_n)=d(Q_n) $ for all $n\ge1$ and  we have just shown that for all $m\ge1$ (and for all $n\ge1$) $d^\wedge_m(D_{2n})=d^\wedge_m(Q_n).$ We note that $|M(D_{2n})|\not=1$ and $|M(Q_n)|=1$ and then we cannot apply Corollary \ref{c:2},  but \eqref{d} and \eqref{q} show, in some sense, that the thesis of Corollary \ref{c:2} is still true. We note also that \eqref{d} and \eqref{q} agree with \cite[Example 3.11]{elr} and  \cite[Section 4]{nr}.

\end{document}